\documentclass[a4paper,12pt]{article}
\usepackage[authoryear,round]{natbib}
\usepackage{amssymb,amsmath,amsthm,graphicx,times,mathrsfs,fullpage}
\usepackage[pdftex]{hyperref}
\hypersetup{plainpages=True, pdfstartview=FitH, bookmarksopen=true,
pdfpagemode=none, colorlinks=true,linkcolor=blue,citecolor=blue}
\usepackage{tikz,extarrows,pifont}
\usetikzlibrary{positioning,shapes,arrows,mindmap}

\def\Binom{{\rm Binom}}
\def\le{\leqslant}
\def\ge{\geqslant}
\def\dd#1{{\,\rm d}#1}
\def\ve{\varepsilon}
\def\tr#1{\left\lfloor #1\right\rfloor}

\numberwithin{equation}{section}
\theoremstyle{plain}
\newtheorem{thm}{Theorem}[section]
\newtheorem{cor}[thm]{Corollary}
\newtheorem{lem}[thm]{Lemma}
\newtheorem{prop}[thm]{Proposition}
\title{Analysis of an exhaustive search algorithm in random graphs
and the $n^{c\log n}$-asymptotics}

\author{Cyril Banderier \\
    LIPN, Institut Galil\'ee\\
    Universit\'e\ Paris 13\\
    93430, Villetaneuse\\ France
\and Hsien-Kuei Hwang\thanks{Part of the work of this author was
done while visiting ISM (Institute of Statistical Mathematics),
Tokyo; he thanks ISM for its hospitality and support.} \\
    Institute of Statistical Science, \\
    Institute of Information Science\\
    Academia Sinica\\
    Taipei 115\\
    Taiwan
\and Vlady Ravelomanana \\
    LIAFA, UMR CNRS 7089\\
    Universit\'e\ Denis Diderot \\
    75205, Paris Cedex 13\\
    France
\and Vytas Zacharovas \\
    Dept. Mathematics \&\ Informatics\\
    Vilnius University\\
  Naugarduko 24, Vilnius\\
    Lithuania}

\date{\today}
\begin{document}
\maketitle


\begin{abstract}

We analyze the cost used by a naive exhaustive search algorithm for
finding a maximum independent set in random graphs under the usual
$\mathscr{G}_{n,p}$-model where each possible edge appears
independently with the same probability $p$. The expected cost turns
out to be of the less common asymptotic order $n^{c\log n}$, which
we explore from several different perspectives. Also we collect many
instances where such an order appears, from algorithmics to
analysis, from probability to algebra. The limiting distribution of
the cost required by the algorithm under a purely idealized random
model is proved to be normal. The approach we develop is of some
generality and is amenable for other graph algorithms.

\end{abstract}


%
%
\section{Introduction}

An \emph{independent set} or \emph{stable set} of a graph $G$ is a
subset of vertices in $G$ no two of which are adjacent. The
\emph{Maximum Independent Set (MIS) Problem} consists in finding an
independent set with the largest cardinality; it is among the first
known NP-hard problems and has become a fundamental, representative,
prototype instance of combinatorial optimization and computational
complexity; see \cite{garey79a}. A large number of algorithms (exact
or approximate, deterministic or randomized), as well as many
applications, have been studied in the literature; see
\cite{bomze99a,fomin10a,woeginger03a} and the references therein for
more information.

The fact that there exist several problems that are essentially
equivalent (including maximum clique and minimum node cover) adds
particularly further dimensions to the algorithmic aspects and
structural richness of the problem. Also worthy of special mention
is the following interesting polynomial formulation (see
\cite{abello01a,harant00a})
\[
    \alpha(G) = \max_{(x_1,\dots,x_n)\in[0,1]^n}
    \left(\sum_{1\le i\le n}x_i-\sum_{(i,j)\in E}
    x_ix_j\right),
\]
where $\alpha(G)$ denotes the cardinality of an MIS of $G$ (or the
stability number) and $E$ is the set of edges of $G$. Such an
expression is easily coded, albeit with an exponential complexity.
The algorithmic, theoretical and practical connections of many other
formulations similar to this one have also been widely discussed;
see \cite{abello01a}.

One simple means to find an MIS of a graph $G$ is the following
exhaustive (or branching or enumerative) algorithm. Start with any
node, say $v$ in $G$. Then either $v$ is in an MIS or it is not.
This leads to the recursive decomposition
\begin{align}\label{alphaG}
    \alpha(G)= \max\biggl\{
    \underbrace{\alpha\left(G \setminus\{v\}\right)}_
    {v\not\in \text{MIS}(G)}, \underbrace{
    1+\alpha\left(G\setminus N^*(v)\right)}_{
    v\in\text{MIS}(G)}\biggr\},
\end{align}
where $\text{MIS}(G)$ denotes an MIS of $G$ and $N^*(v)$ denotes the
union of $v$ and all its neighbors. Such a simple procedure leads to
many refined algorithms in the literature, including alternative
formulations such as backtracking (see \cite{wilf02a}) or branch and
bound (see \cite{fomin10a}).

Tarjan and Trojanowski \cite{tarjan77a} proposed an improved
exhaustive algorithm with worst-case time complexity $O(2^{n/3})$.
Their paper was followed and refined by many since then; see
\cite{bomze99a,woeginger03a} and \cite{fomin10a} for more
information and references. In particular, Chv\'atal
\cite{chvatal77a} generalized Tarjan and Trojanowski's algorithm and
showed \emph{inter alia} that for almost all graphs with $n$ nodes,
a special class of algorithms (which he called order-driven) has time
bound $O(n^{c_0\log n+2})$, where $c_0:=2/\log 2$. He also
characterized exponential algorithms and conjectured that a similar
bound of the form $O(n^{c\log n})$ holds for a wider class of
recursive algorithms for some $c>0$. Pittel \cite{pittel82a} then
refined Chv\'atal's bounds by showing that, under the usual
$\mathscr{G}_{n,p}$-model (namely, \emph{each pair of nodes has the
same probability $p\in(0,1)$ of being connected by an edge, and one
independent of the others}), the cost of Chv\'atal's algorithms
(called $f\!$-driven, more general than order-driven) is bounded
between $n^{(\frac14-\ve)\log_\kappa n}$ and
$n^{(\frac12+\ve)\log_\kappa n}$ with high probability, for any
$\ve>0$, where $q := 1-p$ and $\kappa := 1/q$.

The infrequent scale $n^{c\log n}=e^{c(\log n)^2}$ is central to our
study here and can be seen through several different angles that
will be examined in the following paragraphs. The simplest
algorithmic connection to MIS problem is via the following argument.
It is well-known that for any random graph $G$ (under the
$\mathscr{G}_{n,p}$-model), the value of $\alpha(G)$ is highly
concentrated for fixed $p\in(0,1)$, namely, there exists a sequence
$m_n$ such that $\alpha(G) = m_n$ or $\alpha(G)=m_n+1$ with high
probability; see \cite{bollobas01a}. Asymptotically ($\kappa :=
1/q$),
\[
    m_n = 2\log_{\kappa} n -2\log_{\kappa}\log_{\kappa}n +O(1) .
\]
For more information on this and related estimates, see
\cite{bollobas01a} and the references therein. Thus a simple
randomized (approximate) MIS-finding algorithm consists in examining
all possible
\[
    \binom{n}{m_n}+
    \binom{n}{m_n+1} = O\left(n^{2\log_{\kappa}n}\right)
\]
subsets and determining if at least one of them is independent;
otherwise (which happens with very small probability; see
\cite{bollobas01a}), we resort to exhaustive algorithms such as that
discussed in this paper.

From a different algorithmic viewpoint, Jerrum \cite{jerrum92a}
studied the following Metropolis algorithm for maximum clique.
Sequentially increase the clique, say $K$ by (\emph{i}) choose a
vertex $v$ uniformly at random; (\emph{ii}) if $v\not\in K$ and $v$
is connected to every vertex of $K$, then add $v$ to $K$;
(\emph{iii}) if $v \in K$, then $v$ is subtracted from $K$ with
probability $\Lambda^{-1}$. He proved that for all $\Lambda\ge1$,
there exists an initial state from which the expected time for the
Metropolis process to reach a clique of size at least $(1 + \ve)
\log_\kappa(pn)$ exceeds $n^{\Omega(\log pn)}$. See
\cite{coja-oghlan11a} for an account of more recent developments on
the complexity of the MIS problem.

We aim in this paper at a more precise analysis of the cost used by
the simple recursive, exhaustive algorithm implied by
\eqref{alphaG}. The exact details of the algorithm matter less and
the overall cost is dominated by the total number of recursive
calls, denoted by $X_n$, which is a random variable under the same
$\mathscr{G}_{n,p}$-model. Then the mean value $\mu_n :=
\mathbb{E}(X_n)$ satisfies
\begin{align} \label{mun-rr}
    \mu_n=\underbrace{\mu_{n-1}}_{v\not\in\text{MIS}(G)}
    +\underbrace{\sum_{0\le k<n}
    \pi_{n,k}\mu_k}_{v\in\text{MIS}(G)},
\end{align}
for $n\ge 2$, with the initial conditions $\mu_0=0$ and $\mu_1=1$,
where
\[
    \pi_{n,k} := \mathbb{P}
    (v \text{ has } n-1-k \text{ neighbors})
    = \binom{n-1}{k} p^{n-1-k} q^k.
\]

How fast does $\mu_n$ grow as a function of $n$? (\emph{i}) If $p$
is close to 1, then the graph is very dense and thus the sum in
\eqref{mun-rr} is small (many nodes being removed), so we expect a
polynomial time bound by simple iteration; (\emph{ii}) If $p$ is
sufficiently small, then the second term is large, and we expect an
exponential time bound; (\emph{iii}) What happens for $p$ in
between? In this case the asymptotics of $\mu_n$ turns out to be
nontrivial and we will show that
\begin{align} \label{mun-ae}
    \log \mu_n
    = \frac{\left(\log \frac{n}{\log_\kappa n}\right)^2}
    {2\log \kappa }
    +\left(\tfrac12+\tfrac1{\log \kappa }\right)\log n-\log\log n
    + P_0\left(\log_\kappa \tfrac{n}{\log_\kappa n} \right)+o(1),
\end{align}
where $P_0(t)$ is a bounded, periodic function of period $1$. We
will give a precise expression for $P_0$. Note that
\begin{align} \label{mun-nlogn}
    \frac{\mu_n}{n^{\frac12\log_\kappa n}}
    \asymp \frac{(\log n)^{\frac12\log_\kappa\log n-1
    -\frac{\log\log\kappa}{\log \kappa}}}
    {n^{\log_\kappa \log n -\frac12
    -\frac1{\log\kappa}-\frac{\log\log\kappa}
    {\log\kappa}}}\ll n^{-K}\to0,
\end{align}
for any $K>0$, where the symbol $a_n \asymp b_n$ means that $a_n$
and $b_n$ are asymptotically of the same order. Thus $\mu_n =
o\left(n^{\frac12\log_\kappa n-K}\right)$. On the other hand, the
asymptotic pattern \eqref{mun-ae} is to some extent generic, as we
will see below.

An intuitive way to see why we have the asymptotic form
\eqref{mun-ae} for $\log\mu_n$ is to look at the simpler functional
equation
\begin{align} \label{nux}
    \nu(x) = \nu(x-1) + \nu(qx),
\end{align}
since the binomial distribution is highly concentrated around its
mean value $pn$, and we expect that $\mu_n \approx \nu(n)$ (under
suitable initial conditions). This functional equation and the like
(such as $\nu_n=\nu_{n-1}+\nu_{\tr{qn}}$) has a rich literature.
Most of them are connected to special integer partitions; important
pointers are provided in Encyclopedia of Integer Sequences; see for
example \href{http://oeis.org/A000123}{A000123},
\href{http://oeis.org/A002577}{A002577},
\href{http://oeis.org/A005704}{A005704},
\href{http://oeis.org/A005705}{A005705}, and
\href{http://oeis.org/A005706}{A005706}. In particular, it is
connected to partitions of integers into powers of $\kappa=1/q\ge2$
when $\kappa$ is a positive integer; see
\cite{debruijn48a,fredman74a,mahler40a}. It is known that (under
suitable initial conditions)
\begin{align}\label{log-nux}
    \log \nu(x) = \frac{\left(\log \frac{x}{\log_\kappa x}\right)^2}
    {2\log \kappa }
    +\left(\tfrac12+\tfrac1{\log \kappa }\right)\log x-\log\log x
    + P_1\left(\log_\kappa \tfrac{x}{\log_\kappa x}\right)+o(1),
\end{align}
for large $x$, where $P_1(t)$ is a bounded $1$-periodic function;
see \cite{debruijn48a,dumas96a}. Thus
\[
    |\log\mu_n - \log\nu(n)| =
    \left|P_0\left(\log_\kappa \tfrac{x}{\log_\kappa x}\right)
    -P_1\left(\log_\kappa \tfrac{x}{\log_\kappa x}\right)\right|+o(1).
\]
We see that approximating the binomial distribution in
\eqref{mun-rr} by its mean value
\[
    \mathbb{E}(\mu_{n-1-\Binom(n-1;p)})
    \approx \mu_{n-1-\mathbb{E}(\Binom(n-1;p))}
    \approx \mu_{\tr{qn}}
\]
gives a very precise estimate, where $\Binom(n-1;p)$ denotes a
binomial distribution with parameters $n-1$ and $p$.

An even simpler way to see the dominant order $x^{c\log x}$ is to
approximate \eqref{nux} by the \emph{delay differential equation}
(since $\nu(x)-\nu(x-1)\approx \nu'(x)$ for large $x$)
\begin{align} \label{omegax}
    \omega'(x) = \omega(qx),
\end{align}
which is a special case of the so-called \emph{``pantograph
equations''}
\[
    \omega'(x) = a\omega(qx) + b\omega(x),
\]
originally arising from the study of current collection systems for
electric locomotives; see \cite{iserles93a,kato71a,ockendon71a}.
Since the usual polynomial or exponential functions fail to satisfy
\eqref{omegax}, we try instead a solution of the form $\omega(x) =
x^{c\log x}$; then $c$ should be chosen to satisfy the equation
\[
    x^{1-2c\log\kappa} = 2c e^{c(\log\kappa)^2} \log x.
\]
So we should take $c=1/(2\log\kappa) + O(x^{-1}\log x)$. This gives
the dominant term $\frac{(\log x)^2}{2\log\kappa}$ for
$\log\omega(x)$. More precise asymptotic solutions are thoroughly
discussed in \cite{debruijn53a,kato71a}. In particular, all
solutions of the equation $\omega'(x) = a\omega(qx)$ with $a>0$
satisfies
\begin{align*}
    \log \omega(x)
    &= \frac{\left(\log \frac{x}{\log_\kappa x}\right)^2}
    {2\log \kappa }
    +\left(\tfrac12+\tfrac{1}{\log\kappa}
    +\tfrac{\log a}{\log\kappa}
    \right)\log x -\left(1+
    \tfrac{\log a}{\log\kappa}\right)\log\log x \\
    &\qquad + P_2
    \left(\log_\kappa \tfrac{x}{\log_\kappa x} \right)+o(1),
\end{align*}
for large $x$, where $P_2(t)$ is a bounded $1$-periodic function. We
see once again the generality of the asymptotic pattern
\eqref{mun-ae}.

On the other hand, the function
\[
    \varpi(x) := \exp\left(\frac{\left(\log (x/\sqrt{q})\right)^2
    }{2\log(1/q)}\right)
\]
satisfies the $q$-difference equation
\[
    \varpi(x) = x\varpi(qx),
\]
and is a fundamental factor in the asymptotic theory of
$q$-difference equations; see the two survey papers
\cite{adams31a,di-vizio03a} and the references therein. This
equation will also play an important role in our analysis.

From yet another angle, one easily checks that the series
\begin{align*}
    M(x) := \sum_{j\ge0} \frac{q^{\binom{j}2}}{j!}\,x^j
\end{align*}
satisfies the equation \eqref{omegax}. The largest term occurs, by
simple calculus, at
\[
    j\approx
    \log_\kappa x - \log_\kappa\log_\kappa x + \tfrac12 +o(1),
\]
and, by the analytic approach we use in this paper, we can deduce
that the logarithm of the series is, up to an error of $O(1)$, of
the same asymptotic order as $\log \nu(x)$; see \eqref{log-nux} and
Section~\ref{sec:cr}. The function $M(x)$ arises sporadically in
many different contexts and plays an important r\^{o}le in the
corresponding asymptotic estimates; see below for a list of some
representative references.

A closely related sum arises in the average-case analysis of a
simple backtracking algorithm (see \cite{wilf02a}), which
corresponds to the expected number of independent sets in a random
graph (or, equivalently, the expected number of cliques by
interchanging $q$ and $p$)
\begin{align} \label{Jn}
    J_n := \sum_{1\le j\le n} \binom{n}{j} q^{j(j-1)/2},
\end{align}
see \cite{matula70a,wilf02a}. Wilf \cite{wilf02a} showed that $J_n =
O(n^{\log n})$ when $p=1/2$. While such a crude bound is easily
obtained, the more precise asymptotics of $J_n$ is more involved.
First, it is straightforward to check that $J_n \sim M(n)$ for large
$n$. Second, the approach we develop in this paper can be used to
show that $J_n$ has an asymptotic expansion similar to
\eqref{mun-ae}. Indeed, it is readily checked that $J_n+1$ satisfies
the same recurrence relation as $\mu_n$ with different initial
conditions. So the asymptotics of $J_n$ follows the same pattern
\eqref{mun-ae} as that of $\mu_n$; see Section~\ref{sec:cr} for more
details.

Thus examining all independent sets one after another in the
backtracking style of Wilf \cite{wilf02a} and identifying the one
with the maximum cardinality also leads to an expected $n^{c\log
n}$-complexity.

The diverse aspects we discussed of algorithms or equations leading
to the scale $n^{c\log n}$ are summarized in Figure~\ref{fig:nlogn}.
The bridge connecting the algorithms and the analysis is the
binomial recurrence \eqref{mun-rr} as explained above.

\newcommand{\rewidth}
{
\begin{tikzpicture}[scale=0.65,mindmap,
root concept/.append style={minimum size=2cm},
level 1 concept/.append style={level distance=135,
sibling angle=30,minimum size=1cm}]

\begin{scope}[mindmap,text=white]
\node [concept,left=0.4cm,concept color=green!50!black!90,text width=2cm]
{\footnotesize MIS-finding algorithms
\\ \& $n^{c\log n}$ $(x^{c\log x})$}
child [grow=0,concept color=blue!70]
{node [concept] {\tiny Exhaustive algorithms
$a_{n}-a_{n-1}=\sum\limits_{0\le k<n}
\pi_{n,k} a_k$}}
child [grow=60,concept color=blue!70]
{node [concept] {\tiny Randomized algorithms \\
$\binom{n}{\lfloor c\log n\rfloor}$}}
child [grow=300,concept color=blue!70]
{node [concept] {\tiny Backtracking algorithms
$\sum\limits_{1\le k\le n}\binom{n}{k} q^{\binom{k}{2}}$}}
child [grow=120,concept color=red!75]
{node [concept] {\tiny Mahler's partitions\\
$a_n - a_{n-1} $ $= a_{\lfloor qn\rfloor}$}}
child [grow=180,concept color=red!75]
{node [concept] {\tiny Pantograph equations\\
$f'(x)= af(qx)+bf(x)$}}
child [grow=240,concept color=red!75]
{node [concept] {\tiny $q$-difference equations
$f(x) = xf(qx)$}};
\end{scope}
\end{tikzpicture}
}

\begin{figure}
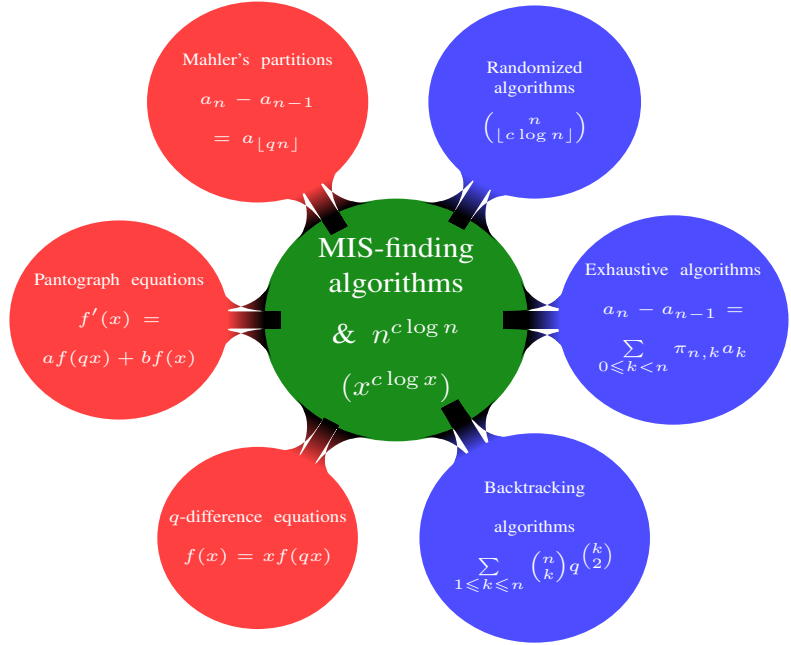

\begin{center}
\resizebox{10.5cm}{8.5cm}{\rewidth}
\end{center}
\label{fig:nlogn}
\caption{\emph{The connection between MIS-finding algorithms
and the scale $n^{c\log n}$ (discrete) or $x^{c\log x}$ (continuous).
The circles on the right-hand side are more algorithmic in nature,
while those on the left-hand side more analytic in nature.}}
\end{figure}

This paper is organized as follows. We derive in the next section an
asymptotic expansion for $\mu_n$ using a purely analytic approach.
The interest of deriving such a precise asymptotic approximation is
at least fourfold.
\begin{description}

\item[\qquad Asymptotics:] It goes much beyond the crude description
$n^{c\log n}$ and provides a more precise description; see
particularly \eqref{mun-nlogn} and its implication mentioned
there. Indeed, few papers in the literature address such an
aspect; see
\cite{debruijn48a,debruijn53a,dumas96a,kato71a,pennington53a,richmond76a}.

\item[\qquad Numerics:] All scales involved in problems of similar
nature here are expressed either in $\log$ or in $\log\log$, making
them more subtle to be identified by numerical simulations. The inherent
periodic functions and the slow convergence further add to the
complications.

\item[\qquad Methodology:] Our approach, different from previous ones
that rely on explicit generating functions in product forms, is
based on the underlying functional equation and is of some
generality; it is akin to some extent to Mahler's analysis in
\cite{mahler40a}.

\item[\qquad Generality:] The asymptotic pattern \eqref{mun-ae} is of
some generality, an aspect already examined in details in
several papers; see for example
\cite{debruijn53a,dumas96a,kato71a}. See also the last section
for a list of diverse contexts where the order $n^{c\log n}$
appears.
\end{description}

Alternative approaches leading to different asymptotic expansions
are discussed in Section~\ref{sec:aa}.

The next curiosity after the expected value is the variance. But due
to strong dependence of the subproblems, the variance is quite
challenging at this stage. We consider instead an idealized
\emph{independent} version of $X_n$ (the total cost of the
exhaustive algorithm implied by \eqref{alphaG}), namely
\begin{align} \label{Yn-rr}
    Y_n \stackrel{d}{=} Y_{n-1} + Y^*_{n-1-\Binom(n-1;p)}
    \qquad(n\ge2),
\end{align}
with $Y_1:=1$ and $Y_0:=0$, where ``$\stackrel{d}{=}$" stands for
equality in distribution, $Y_n^*$ is an identical copy of $Y_n$ and
the two terms on the right-hand side are \emph{independent}. The
original random variable $X_n$ satisfies the same distributional
recurrence but with the two terms ($X_{n-1}$ and $X^*_{n-1-
\Binom(n-1;p)}$) on the right-hand side \emph{dependent}. We expect
that $Y_n$ would provide an insight of the possible stochastic
behavior of $X_n$ although we were unable to evaluate their
difference. We show, by a method of moments, that $Y_n$ \emph{is
asymptotically normally distributed} in addition to deriving an
asymptotic estimate for the variance. Monte Carlo simulations for
$n$ up to a few hundreds show that the limiting distribution of
$X_n$ seems likely to be normal, although the ratio between its
variance and that of $Y_n$ grows like a concave function. But the
sample size $n$ is not large enough to provide more convincing
conclusions from simulations.

Once the asymptotic normality of $Y_n$ is clarified, a natural
question then is the limit law of the random variables (by changing
the underlying binomial to uniform distribution)
\begin{align} \label{Zn-rr}
    Z_n \stackrel{d}{=} Z_{n-1} + Z_{\text{Uniform}(0,n-1)}
    \qquad(n\ge2),
\end{align}
with $Z_0=0$ and $Z_1=1$. In this case, we prove that the mean is
asymptotic to $cn^{-1/4}e^{2\sqrt{n}}$ and the limit law is \emph{no
more normal}. We conclude this paper with a few remarks and a list
of many instances where $n^{c\log n}$ arises, further clarifications
and connections being given elsewhere.

\medskip

\noindent \textbf{Notations}. Throughout this paper, $0<p<1$,
$q=1-p$, and $\kappa = 1/q$.

\section{Expected cost}\label{sec:ec}

We derive asymptotic approximations to $\mu_n$ in this section by an
analytic approach, which is briefly sketched in Figure~\ref{fig-pf}.

\subsection{Preliminaries and main result}
Recall that $X_n$ denotes the cost used by the exhaustive search
algorithm (implied by \eqref{alphaG}) for finding an MIS in a random
graph, and it satisfies the recurrence
\begin{align}
    X_n \stackrel{d}{=} X_{n-1} + X^*_{n-1-\Binom(n-1;p)},
    \label{Xn-rr}
\end{align}
with $X_0=0$ and $X_1=1$, where $X_n^*\stackrel{d}{=} X_n$, and the
two terms on the right-hand side are dependent.

From (\ref{Xn-rr}), we see that the expected value $\mu_n$ of $X_n$
satisfies the recurrence \eqref{mun-rr}. Our analytic approach then
proceeds along the line depicted in Figure~\ref{fig-pf}. While the
approach appears standard (see
\cite{flajolet09a,jacquet98a,szpankowski01a}), the major difference
is that instead of Mellin transform, we need Laplace transform since
the quantity in question is not polynomially bounded. Also the
diverse functional equations are crucial in our analysis, notably
for the purpose of justifying the de-Poissonization, which differs
from previous ones; see \cite{jacquet98a,szpankowski01a}.

\tikzstyle{block} = [rectangle, draw, fill=none, line width = 0.8pt,
text width=5em, rounded corners, minimum height=4em]
\tikzstyle{line} = [draw, -latex',line width = 1pt]
\begin{figure}[!h]
\begin{center}{\footnotesize
\begin{tikzpicture}[auto]
\node [block, text width=4.5cm, rectangle split,rectangle split parts=2,
text centered] (b1)
{
Recurrence relation
\nodepart{second}
{
$\displaystyle\mu_n = \mu_{n-1}+ \sum\nolimits_k \pi_{n,k} \mu_k$
}
};

\node [block, text width=4.5cm, right=4ex of b1, node distance=2.5cm,
rectangle split, rectangle split parts=2, text centered] (b2)
{
Poisson generating function
\nodepart{second}
{
{$\tilde{f}'(z)=\tilde{f}(qz)+e^{-z}$}
}
};

\node [block, text width=4.5cm, below=4ex of b1, node distance=2.5cm,
rectangle split, rectangle split parts=2, text centered] (b3)
{
Poisson-Charlier expansion
\nodepart{second}
{
$\displaystyle
\mu_{n}\sim \tilde{f}(n)-\frac{n}{2}\tilde{f}''(n)$
}
};

\node [block, text width=4.5cm, below=4.3ex of b2, node distance =2.5cm,
rectangle split, rectangle split parts=2, text centered] (b4)
{
Modified Laplace transform
\nodepart{second}
{
$\displaystyle
\tilde{f}^{*}(s)=s\tilde{f}^{*}(qs)+\frac{s}{1+s}$
}
};

\node [block, text width=4.5cm, below =4ex of b3, node distance = 2.5cm,
rectangle split, rectangle split parts=2, text centered] (b5)
{
de-Poissonization
\nodepart{second}
{
$\displaystyle
\mu_{n} =\frac{n!}{2\pi i}\oint z^{-n-1}e^{z}
\tilde{f}(z)\,\text{d} z$
}
};

\node [block, text width=4.5cm, below =4 ex of b4, node distance = 2.5cm,
rectangle split, rectangle split parts=2, text centered] (b6)
{
Inverse transform
\nodepart{second}
{
$\displaystyle\tilde{f}(x)=\frac{1}{2\pi i}
\int \frac{e^{xs}}{s}\tilde{f}^{*}(s)\,\text{d} s$
}
};
\draw [line] (b1) -- (b2);
\draw [line] (b2) -- (b4);
\draw [line] (b4) -- (b6);
\draw [line] (b6) -- (b5);
\draw [line] (b5) -- (b3);
\end{tikzpicture}}
\end{center}\label{fig-pf}
\caption{\emph{Our analytic approach to the asymptotics of
$\mu_n$. Here $\pi_{n,k}:= \binom{n-1}{k}q^{k}p^{n-1-k}$.}}
\end{figure}
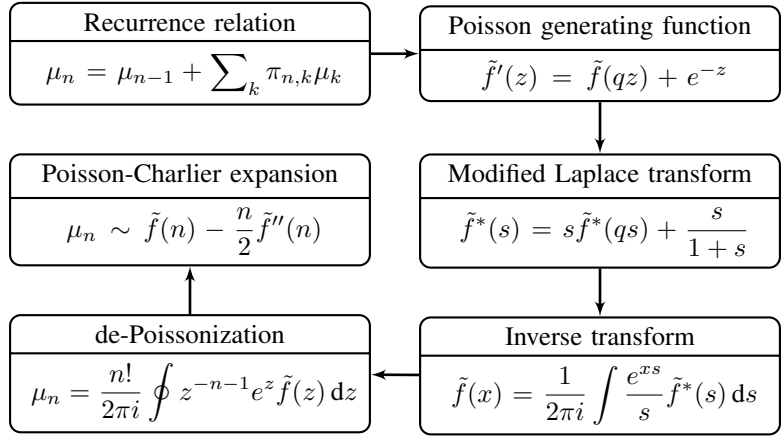

\paragraph{Generating functions (GFs).}
Let $f(z):=\sum_{n\ge0}\mu_n z^n/n!$ denote the exponential GFs of
$\mu_n$. Then $f$ satisfies, by \eqref{mun-rr}, the equation
\[
    f'(z)=1+f(z)+e^{pz}f(qz),
\]
with $f(0)=0$, or, equivalently, denoting by $\tilde{f}(z)
:=e^{-z}f(z)$ the Poisson GF of $\mu_n$,
\begin{align}
    \tilde{f}'(z)=\tilde{f}(qz)+e^{-z},
    \label{tfz-fe}
\end{align}
with $\tilde{f}(0)=0$.

\paragraph{Closed-form expressions.}
Let $\tilde{f}(z)=\sum_{n\ge0} \tilde{\mu}_n z^n/n!$. From the
$q$-differential equation \eqref{tfz-fe}, we derive the recurrence
\[
    \tilde{\mu}_{n+1} = q^n \tilde{\mu}_n +(-1)^n
    \qquad(n\ge1).
\]
By iteration, we then obtain the closed-form expression
\begin{align*}
    \tilde{\mu}_n = \sum_{0\le j<n}(-1)^j
    q^{(n-1-j)(n+j)/2}\qquad(n\ge1).
\end{align*}
Since $f(z) = e^z\tilde{f}(z)$, we then have
\begin{align} \label{mun-closed}
    \mu_n = \sum_{1\le k\le n}\binom{n}{k}
    \sum_{0\le j<k} (-1)^j q^{(k-1-j)(k+j)/2}
    \qquad(n\ge1).
\end{align}
This expression is, although exact, less useful for large $n$; also
its asymptotic behavior remains opaque. See also \eqref{mun-ce2} for
another closed-form expression for $\mu_n$.

\paragraph{Asymptotic approximations.}
Our aim in this section is to derive the following asymptotic
approximation.

\begin{thm} \label{mun-asymp}
The expected cost $\mu_n$ of the exhaustive search on a random graph
satisfies
\begin{equation}\label{mun-ae2}
    \mu_n= \frac{G\left(\log_\kappa\frac{n}
    {\log_\kappa n}\right)}{\sqrt{2\pi}}
    \cdot\frac{n^{1/\log\kappa+1/2}}{\log_\kappa n}
    \,\exp\left(\frac{\left(
    \log \frac{n}{\log_\kappa n}\right)^2}{2\log\kappa}\right)
    \left(1+O\left(\frac{(\log\log n)^2}{\log n}\right)\right),
\end{equation}
as $n\to\infty$, where $G(u)$ is defined by ($\{u\}$ being the
fractional part of $u$)
\[
    G(u) = q^{(\{u\}^2-\{u\})/2}\sum_{j\in\mathbb{Z}}
    \frac{q^{j(j+1)/2}}{1+q^{j-\{u\}}}\, q^{-j\{u\}},
\]
(see \eqref{Gu}) and is a bounded, $1$-periodic function of $u$.
\end{thm}
Note that \eqref{mun-ae2} implies \eqref{mun-ae} with
\[
    P_0(u) = -\tfrac12\log 2\pi -\log \kappa + \log G(u).
\]
Our approach leads indeed to an asymptotic expansion, but we content
ourselves with the statement of \eqref{mun-ae2}; see
\eqref{tfx-ae2}, \eqref{mun-asymp2} and \eqref{tfx-aae}.

The function $f$ (and thus $\tilde{f}$) is an entire function. It
follows immediately that we have the identity (see \cite{hwang10a})
\[
    \mu_n = \sum_{j\ge0} \frac{\tilde{f}^{(j)}(n)}{j!}
    \,\tau_j(n),
\]
(referred to as the Poisson-Charlier expansion in \cite{hwang10a})
where the $\tau_j(n)$'s are polynomials of $n$ of degree $\tr{n/2}$;
see \eqref{tau-jn}. See also \cite{jacquet98a} for different
representations. However, the hard part is often to
justify the \emph{asymptotic nature} of the expansion, namely,
\[
    \mu_n = \sum_{0\le j<J} \frac{\tilde{f}^{(j)}(n)}{j!}
    \,\tau_j(n) + O\left(n^{\tr{J/2}} \tilde{f}^{(J)}(n)\right),
\]
for $J=2,3,\dots$. In particular, the first-order asymptotic
equivalent ``$\mu_n \sim \tilde{f}(n)$" is often called the
\emph{Poisson heuristic}. Thus the asymptotics of $\mu_n$ is reduced
to that of $\tilde{f}(x)$ once we justify the asymptotic nature of
the expansion. Of special mention is that, unlike almost all papers
in the literature, we need only the asymptotic behavior of
$\tilde{f}(x)$ for \emph{real values of $x$}, all analysis involving
complex parameters being carefully handled by the corresponding
functional equation.

We will derive an asymptotic expansion for $\tilde{f}(x)$ for large
real $x$ by Laplace transform techniques and suitable manipulation
of the saddle-point method, and then bridge the asymptotics of
$\mu_n$ and $\tilde{f}(n)$ by a variant of the saddle-point method
(or de-Poissonization procedure; see \cite{jacquet98a}); see
Figure~\ref{fig-pf} for a sketch of our proof.

\subsection{Asymptotics of $\tilde{f}(x)$}

We derive an asymptotic expansion for $\tilde{f}(x)$ in this
subsection.

\paragraph{Modified Laplace transform.}
For technical convenience, consider the modified Laplace transform
\begin{align*}
    \tilde{f}^\star(s):=\frac{1}{s}\int_0^\infty
    e^{-x/s}\tilde{f}(x)\dd x.
\end{align*}
Note that this use of the Laplace transform differs from the usual
one by a factor $1/s$ and by a change of variables $s\mapsto 1/s$.
Also the use of the exponential GF coupling with this Laplace
transform is equivalent to considering the ordinary GF of $\mu_n$;
see Section~\ref{sec:ovse} for more information.

Then the functional-differential equation (\ref{tfz-fe}) translates
into the following functional equation for $\tilde{f}^\star$
\begin{align} \label{f-star-fe}
    \tilde{f}^\star(s)=s\tilde{f}^\star(qs)+\frac{s}{1+s},
\end{align}
for $\Re(s)>0$.

Iterating the equation \eqref{f-star-fe} indefinitely, we get
\begin{align}\label{Ls-exact}
    \tilde{f}^\star(s)=\sum_{j\ge 0}
    \frac{q^{j(j+1)/2}}{1+q^js}\,s^{j+1}.
\end{align}
We will approximate $\tilde{f}^\star(s)$ for large $s$ by means of
the function
\[
    F(s)=\sum_{-\infty<j<\infty}
    \frac{q^{j(j+1)/2}}{1+q^js}\,s^{j+1},
\]
because adding terms of the form $s^{-j}$, $j\ge0$, does not alter
the asymptotic order of both functions.

\begin{lem} For $x>1$, we have
\begin{align}\label{L-f}
    F(x) =x^{1/2}\exp\left( \frac{(\log x)^2}{2\log\kappa}
    \right)G\left(\log_\kappa x \right),
\end{align}
where
\begin{align}
    G(u):=q^{(\{u\}^2+\{u\})/2}F\left(q^{-\{u\}}\right)
    \label{Gu}
\end{align}
is a continuous, positive, periodic function with period $1$.
\end{lem}
\begin{proof}
One can easily check that $F(s)$ satisfies a functional equation
similar to that of Jacobi's theta functions
\begin{equation}\label{L-fe}
    F(s)=sF(qs) \qquad(s\in\mathbb{C}).
\end{equation}
Iterating $N$ times this functional equation, we obtain
\begin{equation}\label{L-fe1}
    F(s)=q^{N(N-1)/2}s^NF\left(q^Ns\right)
    \qquad(s\in\mathbb{C}).
\end{equation}
Assume $x>1$. Take
\[
    N=\left\lfloor\log_\kappa x\right\rfloor
    =\log_\kappa x+\eta,
\]
where $\eta=-\left\{\log_\kappa x\right\}$. Then we have
\begin{align*}
    F(x)&=\exp\left(\frac{N(N-1)}{2}\,\log q+N\log x
    \right)F\left(e^{N\log q+\log x}\right)\\
    &=\exp\left( \frac{(\log x)^2}{2\log\kappa} +\frac{\log x}{2}
    +\frac{\eta(\eta-1)}{2}\log q\right)
    F\left(e^{\eta\log q}\right)\\
    &=q^{(\eta^2-\eta)/2}x^{1/2}
    \exp\left( \frac{(\log x)^2}{2\log\kappa}
    \right) F\left(e^{\eta\log q}\right),
\end{align*}
which, together with the functional equation $F(1/q)=F(1)/q$
(or $G(u+1)=G(u)$), proves the lemma.
\end{proof}

\paragraph{Asymptotic expansion of $\tilde{f}(x)$: saddle-point
method} By the inversion formula, we have
\begin{equation}\label{fx-Lif}
    \tilde{f}(x)=\frac{1}{2\pi i}
    \int_{r-i\infty}^{r+i\infty}
    \frac{e^{xs}}{s}\tilde{f}^\star\left(\frac{1}{s}\right)\dd s,
\end{equation}
where $r>0$ is a small number whose value will be specified later.
We now derive a few estimates for $\tilde{f}^\star(s)$.
\begin{lem} \emph{(i)} If $r>0$ and $|t|\ge 1$, then
\begin{align} \label{f1}
    \tilde{f}^\star\left(\frac{1}{r+it}\right)
    =O\left(\frac{1}{|t|}\right);
\end{align}
\emph{(ii)} if $0<r\le1$ and $|t|\le1$, then
\begin{equation}\label{f2}
    \tilde{f}^\star\left(\frac{1}{r+it}\right)
    =F\left(\frac{1}{r+it}\right)+O(1);
\end{equation}
\emph{(iii)} if $r>0$ and $c_mr\le |t|\le 1$,
where $c_m := \sqrt{q^{-2m}-1}$, $m\ge1$, then
\begin{align} \label{f3}
    \tilde{f}^\star\left(\frac{1}{r +it}\right)
    =O\left(r^m q^{\binom{m}2} F\left(\frac1r\right)\right).
\end{align}
\end{lem}
\begin{proof}
First, \eqref{f1} follows from (\ref{Ls-exact}). For the estimate
\eqref{f2}, we observe that
\[
    \left|\frac1{1+s q^j}\right|
    \le \min\{q^{-j}|s|^{-1},1\} \qquad(\Re(s)\ge0).
\]
Then
\[
    \tilde{f}^\star(s)=F(s)+O\left(|s|^{-1}\right),
\]
for $\Re(s)\ge 0$ and $|s|\ge c>0$. Also for $r>0$
\[
    \Re\left(\frac{1}{r+it}\right)
    =\frac{r}{r^2+t^2}>0;
\]
and, for $|t|\le 1$ and $0<r\le 1$
\[
    \frac{1}{|r+it|}\ge \frac{1}{\sqrt{2}}.
\]
From these two estimates, we then deduce \eqref{f2}.

On the other hand if $\Re(s)\ge0$,  then
\[
    |\tilde{f}^\star(s)|\le  \sum_{j\ge0} q^{j(j+1)/2}
    |s|^{j+1}\le \vartheta(|s|),
\]
where
\[
    \vartheta(x) :=\sum_{-\infty<j<\infty} {q^{j(j-1)/2}}x^{j}.
\]
It is easily checked that $\vartheta(x)$ satisfies the same
functional equation (\ref{L-fe}) as $F(x)$, namely,
\[
    \vartheta(x)=x\vartheta(qx).
\]
Thus, by the same arguments used for $F(x)$, we have, for $x>1$,
\[
    \vartheta(x)=x^{1/2}\exp\left( \frac{(\log x)^2}{2\log\kappa}
    \right)g(\log_\kappa x),
\]
where $g(x)$ is a continuous, bounded, periodic function. Comparing
this expression with (\ref{L-f}) for $F(x)$, we conclude that
$\vartheta(x)= O(F(x))$ for $x\ge1$.

Let $c_m:=\sqrt{q^{-2m}-1}$, $m>1$. Then, for $0<r<1$,
\begin{align*}
    \max_{c_mr\le |t|\le 1}
    \left|\tilde{f}^\star\left(\frac{1}{r +it}\right)\right|
    &\le  \max_{c_mr\le |t|\le 1} \left|\vartheta
    \left(\frac{1}{\sqrt{r^2+t^2}}\right)\right| \\
    &=\vartheta(q^m/r) \\
    &=r^mq^{m(m-1)/2}\vartheta(1/r)
    \\ &=O\left(r^m q^{\binom{m}2} F(1/r)\right).
\end{align*}
This proves \eqref{f3} and the lemma.
\end{proof}

By splitting the integral in \eqref{fx-Lif} into three ranges
$|t|\le c_mr$, $c_mr<|t|\le 1$, and $|t|>1$, and then applying the
estimates \eqref{f1} and \eqref{f3}, we deduce that
\begin{align}\label{tfx-ae1}
    \tilde{f}(x)e^{-xr}= I_r(x)
    +O\left(r^{m-1} q^{\binom{m}{2}} F(1/r)+1\right),
\end{align}
where
\[
    I_r(x) := \frac{1}{2\pi}\int_{-c_mr}^{c_mr}
    \frac{e^{ixt}}{r+it}\,
    F\left(\frac{1}{r+it}\right)\dd t.
\]
It remains to evaluate more precisely the integral $I_r(x)$ by
the saddle-point method.

We now take
\[
    N=\left\lfloor \log_\kappa(1/r)
    \right\rfloor=\log_\kappa(1/r)+\eta,
\]
where $\eta=-\{\log_\kappa(1/r)\}$. Applying the functional equation
(\ref{L-fe1}) with $s=1/(r+it)$, we get
\[
    I_r(x)=\frac{1}{2\pi}\int_{-c_mr}^{c_mr}
    \frac{e^{ixt}q^{N(N-1)/2}}{(r+it)^{N+1}}
    F\left(\frac{r q^{\eta}}{r+it}\right)
    \dd t.
\]
By the relation
\[
    F(1/r)=q^{N(N-1)/2}r^{-N}F(q^{\eta}),
\]
we then have
\[
\begin{split}
    I_r(x) &=\frac{F(1/r)}{2\pi r}
    \int_{-c_mr}^{c_mr}{e^{ixt}}
    \left(\frac{r}{r+it}\right)^{N+1}
    \frac{F (r q^{\eta}/(
    r+it))}{F(q^{\eta})}\dd t\\
    &=\frac{F(1/r)e^{xr}}{2\pi}\int_{-c_m}^{c_m}
    {e^{irxt}}\left(\frac{1}{1+it}\right)^{N+1}
    \frac{F(q^{\eta}/(1+it))}
    {F(q^{\eta})}\dd t\\
    &=\frac{F(1/r)e^{xr}}{2\pi}
    \int_{-c_m}^{c_m} e^{-xrt^2/2}
    H(t) \dd t,
\end{split}
\]
where
\[
    H(t) := e^{xr(it-\log(1+it)+t^2/2)}
    \frac{F(q^{\eta}/(1+it))}
    {(1+it)^{1+\eta}F(q^{\eta})}.
\]
We now choose $r=r(x)>0$ to be the approximate saddle-point such
that
\begin{align}
    \frac1r\log\frac1r = x\log\kappa. \label{eq-saddle}
\end{align}
Note that $r$ can be expressed in terms of the Lambert-W function
(principal solution of the equation $W(x)e^{W(x)}=x$) as
\[
    r = \frac{W(x\log\kappa)}{x\log\kappa};
\]
thus $\log(1/r)=W(x\log\kappa)$. Asymptotically,
\begin{align}\label{Wx-asymp}
    W(x) = \log x - \log\log x + \frac{\log\log x}{\log x}
    +\frac{(\log\log x)^2-2\log\log x}{2(\log x)^2} +
    O\left(\frac{(\log\log x)^3}{(\log x)^3}\right),
\end{align}
as $x\to\infty$; see \cite{corless96a}.

Since $m>1$ is arbitrary and $r\asymp x^{-1}\log x$, the
relation (\ref{tfx-ae1}) is an asymptotic approximation,
albeit less explicit.

To derive a more explicit expansion, we first observe that
\[
    e^{xr} F(1/r) = r^{-1/\log\kappa-1/2}
    e^{(\log(1/r))^2/(2\log\kappa)}
    G(\log_\kappa(1/r)),
\]
by (\ref{L-f}) and (\ref{eq-saddle}). Then what remains is standard
(see \cite{flajolet09a}): evaluating the integral in (\ref{tfx-ae1})
by Laplace's method (a change of variable $t \mapsto t/\sqrt{xr}$
followed by an asymptotic expansion of $H(t/\sqrt{xr})$ for large
$xr$ and then an integration term by term), and we obtain the
following expansion.

\begin{prop} \label{mu-asymp}
With $r$ given by (\ref{eq-saddle}), $\tilde{f}(x)$ satisfies
\begin{equation}\label{tfx-ae2}
    \tilde{f}(x) \sim\frac{e^{(\log(1/r))^2/(2\log\kappa)}
    G(\log_\kappa(1/r))}{r^{1/\log\kappa+1/2}
    \sqrt{{2\pi\log_\kappa(1/r)}}}\left(1+
    \sum_{j\ge1} \phi_j(\log_\kappa(1/r))
    (\log_\kappa(1/r))^{-j}\right),
\end{equation}
as $x\to\infty$, where $G$ is given in (\ref{Gu}) and the
$\phi_j(u)$'s are bounded, $1$-periodic functions of $u$ involving
the derivatives of $F\left(q^{-\{u\}}\right)$.
\end{prop}
In particular,
\[
    \phi_1(u) = -\left(\frac1{12} -\frac{\{u\}(1-\{u\})}{2}
    +\frac{(1-\{u\})q^{-\{u\}} F'\left(q^{-\{u\}}\right)}
    {F\left(q^{-\{u\}}\right)}+\frac{q^{-2\{u\}}
    F''\left(q^{-\{u\}}\right)}{2F\left(q^{-\{u\}}\right)}\right).
\]

By using (\ref{Wx-asymp}), the leading term in (\ref{tfx-ae2}) can
be expressed completely in terms of $\log x$ as follows.
\begin{cor} As $x\to\infty$, $\tilde{f}(x)$ satisfies
\begin{align}\label{tfx-ae3}
    \tilde{f}(x)= \frac{G\left(
    \log_\kappa\frac{x}{\log_\kappa x}\right)}{\sqrt{2\pi}}\cdot
    \frac{x^{1/\log\kappa+1/2}}
    {\log_\kappa x}\exp\left(\frac{\left(
    \log \frac{x}{\log_\kappa x}\right)^2}{2\log\kappa}\right)
    \left(1+O\left(\frac{(\log\log x)^2}{\log x}\right)\right).
\end{align}
\end{cor}
This is nothing but \eqref{mun-asymp} with $n$ there replaced by
$x$.

As another consequence, we see, by (\ref{tfz-fe}) and
(\ref{tfx-ae3}), that
\[
    \frac{\tilde{f}'(x)}{\tilde{f}(x)} \sim
    \frac{\tilde{f}(qx)}{\tilde{f}(x)} \sim
    \frac{\log_\kappa x}{x}.
\]
More generally, we have the following asymptotic relations for
$\tilde{f}^{(j)}(x)$ and $\tilde{f}(q^jx)$.
\begin{cor} \label{lmm-ratio-f} For $j\ge1$
\begin{align}
    \frac{\tilde{f}^{(j)}(x)}{\tilde{f}(x)} &\sim
    \left(\frac{\log_\kappa x}{x}\right)^{j}\label{ratio-f1}\\
    \frac{\tilde{f}(q^jx)}{\tilde{f}(x)} &\sim q^{-j(j-1)/2}
    \left(\frac{\log_\kappa x}{x}\right)^{j}.\label{ratio-f2}
\end{align}
\end{cor}
Note that (\ref{ratio-f1}) also follows easily from the integral
representation
\[
    \tilde{f}^{(j)}(x) = \frac1{2\pi i} \int_{r-i\infty}
    ^{r+i\infty} \frac{e^{xs}}{s^{j-1}}\tilde{f}^\star
    \left(\frac1s\right) \dd s,
\]
and exactly the same arguments used above.

\subsection{Asymptotics of $\mu_n$}

We first derive a simple lemma for the ratio $f(x+y)/f(x)$ when $y$
is not too large by using (\ref{ratio-f1}).
\begin{lem}\label{lmm-ratio-f2} Assume $x>1$. If $|y|=o(x/\log x)$,
then
\begin{align}
    \frac{\tilde{f}(x+y)}{\tilde{f}(x)}
    =1+O\left(\frac{|y|\log x}{x}\right).\label{ratio-ff}
\end{align}
\end{lem}
\begin{proof}
By (\ref{ratio-f1}), we have
\[
\begin{split}
    \log\frac{\tilde{f}(x+y)}{\tilde{f}(x)}
    &=y\int_0^1\frac{\tilde{f}'(x+yt)}
    {\tilde{f}(x+yt)}\dd t\\
    &=yO\left(\int_0^1\frac{\log|x+yt|}
    {|x+yt|}\dd t\right) \\
    &=O\left(\frac{|y|\log|x|}{|x|}
    \right),
\end{split}
\]
from which (\ref{ratio-ff}) follows.
\end{proof}

\begin{thm} The expected cost used by the exhaustive search algorithm
satisfies the asymptotic expansion
\begin{align}\label{mun-asymp2}
    \mu_n \sim \tilde{f}(n) + \sum_{j\ge2}
    \frac{\tilde{f}^{(j)}(n)}{j!}\,\tau_j(n),
\end{align}
where $\tau_j(n)$ is a (Charlier) polynomial in $n$ of degree
$\lfloor j/2\rfloor$ defined by
\begin{align}\label{tau-jn}
    \tau_j(n) := \sum_{0\le \ell \le j}
    \binom{j}{\ell}(-1)^\ell \frac{n! n^\ell}{(n-k+\ell)!}
    \qquad(j=0,1,\dots).
\end{align}
\end{thm}
In particular, $\tau_0(n)=1$, $\tau_1(n) = 0$, $\tau_2(n) = -n$,
$\tau_3(n) = 2n$, and $\tau_4(n) = 3n^2-6n$. Thus, by
\eqref{tfx-ae2} and \eqref{ratio-f1},
\[
    \mu_n = \tilde{f}(n) \left(1+O\left(n^{-1}(\log n)^2\right)
    \right) ,
\]
which proves Theorem~\ref{mun-asymp}.

\begin{proof}
For simplicity, we prove only the following estimate
\begin{align} \label{mun-asymp3}
    \mu_n=\tilde{f}(n)-\frac{n}{2}\,\tilde{f}''(n)
    +O\left( n^{-2}(\log n)^4\tilde{f}(n)\right).
\end{align}
The same method of proof easily extends to the proof of
\eqref{mun-asymp2}.

We start with the Taylor expansion of $\tilde{f}(z)$ at $z=n$ to the
fourth order
\begin{align}\label{fz-taylor}
    \tilde{f}(z)=\tilde{f}(n)+\tilde{f}'(n)(z-n)
    +\frac{\tilde{f}''(n)}{2!}(z-n)^2+
    \frac{\tilde{f}'''(n)}{3!}(z-n)^3+(z-n)^4R(z),
\end{align}
where
\[
    R(z)=\frac{1}{3!}\int_0^1\tilde{f}^{(4)}
    \bigl(n+(z-n)t\bigr)(1-t)^3\dd t.
\]
By applying successively the equation (\ref{tfz-fe}), we get
\[
    \tilde{f}^{(4)}(z)=-e^{-z}+q^3e^{-qz}-q^5e^{-q^2z}
    +q^6e^{-q^3z}+q^6\tilde{f}(q^4z).
\]
It follows that
\begin{align*}
    \left|R\left(ne^{i\theta}\right)\right|
    &\le\int_0^1\bigl|\tilde{f}^{(4)}
    \bigl(n+n(e^{i\theta}-1)t\bigr)\bigr|\dd t  \\
    &=O\left(e^{-n\cos \theta}+e^{-q^3n\cos \theta}+
    \int_0^1\bigl|\tilde{f}\bigl(q^4n+q^4n
    (e^{i\theta}-1)t\bigr)\bigr|\dd t \right),
\end{align*}
for $|\theta|\le\pi$. Replacing first $\tilde{f}(z)$ inside the
integral by $e^{-z}f(z)$, using the inequality $|f(z)|\le f(|z|)$
and then substituting back $f(q^4n)$ by $e^{q^4n}\tilde{f}(q^4n)$,
we then have
\begin{align}
    \left|R\left(ne^{i\theta}\right)\right|
    &=O\left(e^{-q^3n\cos \theta}+
    f(q^4n)\int_0^1\bigl|e^{-q^4n-q^4n(e^{i\theta}-1)t}\bigr|\dd t
    \right) \notag \\ &=O\left(e^{-q^3n\cos \theta}+
    \tilde{f}(q^4n)\int_0^1e^{q^4n(1-\cos \theta)t}\dd t \right)
    \notag\\
    &=O\left(e^{-q^3n\cos \theta}+
    \tilde{f}(q^4n)e^{q^4n(1-\cos \theta)}\right), \label{Rz-est}
\end{align}
uniformly for $|\theta|\le\pi$. By Cauchy's integral formula and
(\ref{fz-taylor}), we have
\[
\begin{split}
    \mu_n&=\frac{n!}{2\pi i}\oint_{|z|=n} z^{-n-1} e^z
    \tilde{f}(z)\dd z \\
    &=\frac{n!}{2\pi i}\oint_{|z|=n}z^{-n-1} e^z
    \left(\tilde{f}(n)+\frac{\tilde{f}'(n)}{1!}(z-n)
    +\frac{\tilde{f}''(n)}{2!}(z-n)^2+
    \frac{\tilde{f}'''(n)}{3!}(z-n)^3\right)\dd z\\
    &\qquad +R_n \\
    &= \tilde{f}(n)-\frac{n}{2}\tilde{f}''(n)
    +\frac{n}{3}\tilde{f}'''(n)+R_n,
\end{split}
\]
where
\[
    R_n := \frac{n!}{2\pi i}\oint_{|z|=n}
    z^{-n-1} e^z(z-n)^4R(z)\dd z.
\]
By the estimate (\ref{Rz-est}) for $R(z)$, we have
\[
\begin{split}
    R_n &=O\left(n!n^{4-n}\int_{-\pi}^\pi\theta^4
    e^{n\cos\theta}|R(ne^{i\theta})|\dd \theta\right) \\
    &=O\left(n!n^{4-n}\int_{-\pi}^\pi
    \theta^4e^{n\cos\theta}\left(e^{-q^3n\cos \theta}+
    \tilde{f}(q^4n)e^{q^4n(1-\cos\theta)}\right)\dd\theta\right)
    \\ &=O\left(n! n^{4-n}\int_{-\pi}^\pi\theta^4
    e^{n(1-q^3)\cos \theta}\dd\theta +n!
    \tilde{f}(q^4n)n^{4-n}e^n\int_{-\pi}^\pi
    \theta^4e^{-(1-q^4)n(1-\cos \theta)}\dd\theta
    \right)\\ &=O\left(n!n^{-n+3/2}e^{(1-q^3)n}
    +n!e^nn^{-n+3/2}\tilde{f}(q^4n)\right)\\
    &=O\left(n^2e^{-q^3n} + n^2 \tilde{f}(q^4n)\right)\\
    &=O\left(n^{-2}(\log n)^4\tilde{f}(n)\right),
\end{split}
\]
by (\ref{ratio-f2}). Note that again by (\ref{ratio-f1})
\[
    n\tilde{f}'''(n)=O\left( n^{-2}(\log n)^3\tilde{f}(n)\right),
\]
so this error bound is absorbed in $O(\tilde{f}(n) n^{-2}(\log
n)^4)$. This proves \eqref{mun-asymp3}.
\end{proof}

\section{Alternative expansions and approaches}\label{sec:aa}

We discuss in this section other possible approaches to the
asymptotic expansions we derived above.

\subsection{An alternative expansion for $\tilde{f}(x)$}
We begin with an alternative asymptotic expansion for
$\tilde{f}(x)$, starting from the integral representation
\eqref{fx-Lif}, which, as showed above, can be approximated by
\[
    \tilde{f}(x) = \frac1{2\pi i}\int_{r-i\infty}^{r+i\infty}
    \frac{e^{xs}}{s}\,F\left(\frac1s\right)\dd{s}+O(1)
\]
For simplicity, we will write this as
\[
    \tilde{f}(x) \simeq \frac1{2\pi i}\int_{r-i\infty}^{r+i\infty}
    \frac{e^{xs}}{s}\,F\left(\frac1s\right)\dd{s}.
\]
Now we use the same $N=\tr{\log_\kappa (1/r)}
=\log_\kappa(1/r)-\eta$ and
\[
    F\left(\frac1s\right) = q^{N(N-1)/2}s^{-N}
    F\left(\frac{q^N}s\right),
\]
so that
\begin{align} \label{fx-ir}
    \tilde{f}(x) \simeq \frac{q^{\binom{N}2}}{2\pi i}
    \int_{r-i\infty}^{r+i\infty}
    \frac{e^{xs}}{s^{N+1}}\,F\left(\frac{q^N}s\right)\dd{s}.
\end{align}
Now instead of expanding $F(q^N/(r+it))$ at $t=0$, we expand
$F(q^N/s)$ at $s=r$, giving
\[
    F\left(\frac{q^N}s\right) = F\left(\frac{q^N}r-
    \frac{q^N}{r}\left(1-\frac rs\right)\right)
    = \sum_{m\ge0}\frac{(-1)^mQ^m}{m!}\,
    F_m \left(1-\frac rs\right)^m,
\]
where $Q := q^N/r=q^{-\{\log_\kappa (1/r)\}}$ and $F_j$ denotes
$F^{(j)}(Q)$. Substituting this expansion into the integral
representation \eqref{fx-ir} and then integrating term-by-term, we
obtain
\begin{align}
    \tilde{f}(x)q^{-\binom{N}2}
    &\simeq \sum_{m\ge0}\frac{(-1)^mQ^m}{m!}\,
    F_m \cdot\frac1{2\pi i}\int_{r-i\infty}^{r+i\infty}
    \frac{e^{xs}}{s^{N+1}}\,\left(1-\frac rs\right)^m\dd{s}
    \nonumber\\
    &= \frac{x^N}{N!}\sum_{m\ge0}\frac{(-1)^mQ^m}{m!}\,
    F_m T_m(N), \label{tfx-aea}
\end{align}
where, by the integral representation for Gamma function (see
\cite{flajolet09a}),
\begin{align*}
    T_m(N) &:= \frac1{2\pi i}\int_{r-i\infty}^{r+i\infty}
    \frac{e^{xs}}{s^{N+1}}\,\left(1-\frac rs\right)^m\dd{s}
    \nonumber\\  &=\sum_{0\le j\le m}
    \binom{m}{j}(-r)^j \frac{N!x^j}{(N+j)!}.
\end{align*}
For computational purposes, it is preferable to use the recurrence
\[
    T_m(N) = T_{m-1}(N) - \frac{rx}{N+1}\,T_{m-1}(N+1).
\]
The value of $r$ is arbitrary up to now. If we take $r=N/x$, then
\[
    T_m(N) := \sum_{0\le j\le m}
    \binom{m}{j}(-1)^j \frac{N!N^j}{(N+j)!}.
\]
Note that $|T_m(N)|\asymp N^{-\lceil m/2\rceil}$. In particular,
\[
    T_0(N) = 1,\; T_1(N) = \frac1{N+1},\;
    T_3(N) = -\frac{N-2}{(N+1)(N+2)},\; \cdots.
\]
Since $q^N/r$ remains bounded, we can regroup the terms and get an
asymptotic expansion in terms of increasing powers of $N^{-1}$, the
first few terms being given as follows
\begin{align*}
\begin{split}
    \frac{\tilde{f}(x)}{q^{\binom{N}2}x^N/N!}
    &\simeq F_0 - \frac{Q(2F_1+F_2Q)}{2N}
    + \frac{Q(3F_4Q^3+28F_3Q^2+60F_2Q+24F_1)}
    {24N^2} \\
    &\quad - \frac{Q(F_6Q^5+22F_5Q^4+152F_4Q^3
    +384F_3Q^2+312F_2Q+48F_1)}{48N^3}\\
    &\quad+\cdots.
\end{split}
\end{align*}

On the other hand, if we choose $r=(N+1)/x$, then $T_1(N)=0$ and
\[
    T_0(N) = 1,\; T_2(N) = -\frac1{N+2},\;
    T_3(N) = -\frac{4}{(N+2)(N+3)}, \;\cdots,
\]
so that
\begin{align*}
\begin{split}
    \frac{\tilde{f}(x)}{q^{\binom{N}2}x^N/N!}
    &\simeq F_0 - \frac{F_2Q^2}{2(N+2)}
    + \frac{Q^3(3F_4Q+16F_3)}{24(N+2)^2} \\
    &\quad - \frac{Q^3(F_6Q^3+16F_5Q^2+60F_4Q
    +32F_3)}{48(N+2)^3}+\cdots.
\end{split}
\end{align*}
While $|T_m(N)|\asymp N^{-\lceil m/2\rceil}$ for $m\ge2$ as in the
case of $r=N/x$, this is a better expansion because the first term
incorporates more information.

The more transparent expansion \eqref{tfx-aea} is \emph{a priori} a
\emph{formal} one whose asymptotic nature can be easily justified by
the same local analysis as above, details being omitted here. We
summarize the analysis in the following theorem.
\begin{thm} The Poisson generating function of $\mu_n$ satisfies the
asymptotic expansion
\begin{align} \label{tfx-aae}
    \tilde{f}(x) \sim q^{\binom{N}{2}}
    \frac{x^N}{N!}\sum_{m\ge0}\frac{(-1)^mQ^m}{m!}\,
    F^{(m)}(Q) T_m(N),
\end{align}
where
$N=\tr{\log_\kappa (1/r)}=\log_\kappa(1/r)-\eta$, $r := N/x$,
$Q:= q^{-\log_\kappa(1/r)}$ and $T_m(N)$ is defined by
\[
    T_m(N) := \sum_{0\le j\le m}
    \binom{m}{j}(-1)^j \frac{N!(N+1)^j}{(N+j)!}.
\]
\end{thm}
Straightforward calculations give (when $r=N/x$)
\begin{align*}
    \log\left(q^{\binom{N}2}\frac{x^N}{N!}\right)
    &= \frac{\left(\log\frac{x}{\log_\kappa x}\right)^2}
    {2\log\kappa}
    +\left(\frac1{\log\kappa}+\frac12\right)\log x
    -\log\log x  \\ &\quad -\frac12\log 2\pi-
    \frac{\eta^2+\eta}2 + O\left(\frac{(\log\log x)^2}{\log x}\right),
\end{align*}
consistent with what we proved in \eqref{tfx-ae3} via directly
applying the saddle-point method. For similar types of
approximation, see \cite{heller71a,mahler40a}.

\subsection{Exponential GFs vs ordinary GFs} \label{sec:ovse}

The different forms of the GFs of the sequence $\mu_n$ have several
interesting features which we now briefly explore.

Instead of $\tilde{f}^\star(s)$, we start with considering the usual
Laplace transform of $\tilde{f}(z)$
\[
    \mathscr{L}(s) = \int_0^\infty e^{-xs} \tilde{f}(x) \dd x,
\]
which, by (\ref{Ls-exact}), satisfies
\[
    \mathscr{L}(s) = \sum_{j\ge0} \frac{q^{\binom{j+1}{2}}}
    {s^{j+1}(s+q^j)}.
\]
By inverting this series, we obtain
\[
    \tilde{f}(z) = \sum_{j\ge0} \frac{q^{\binom{j+1}{2}}}{j!}
    z^{j+1}\int_0^1 e^{-q^j uz} (1-u)^j \dd u.
\]
From this exact expression, we deduce not only the exact expression
\eqref{mun-closed} but also the following one (by multiplying both
sides by $e^z$ and then expanding)
\begin{align} \label{mun-ce2}
    \mu_n = n\sum_{0\le j<n} \binom{n-1}{j}
    q^{\binom{j+1}{2}}\sum_{0\le \ell < n-j}
    \binom{n-1-j}{\ell} \frac{q^{j\ell}(1-q^j)^{n-1-j-\ell}}
    {j+\ell+1},
\end{align}
where all terms are now positive; compare \eqref{mun-closed}. But
this expression and \eqref{mun-closed} are less useful for numerical
purposes for large $n$.

On the other hand, the consideration of our $\tilde{f}^\star(s)$
bridges essentially EGF and OGF of $\mu_n$. Indeed,
\begin{align*}
    \tilde{f}^\star(s) &= \frac1s\int_0^\infty e^{-x-x/s}
    \sum_{n\ge0} \frac{\mu_n}{n!}\,x^n \dd x \\
    &=\frac1{1+s}\sum_{n\ge0} \mu_n \left(\frac{s}{1+s}\right)^n,
\end{align*}
which is essentially the Euler transform of the OGF; see
\cite{flajolet92a}.

Our proofs given above rely strongly on the use of EGF, but the use
of OGF works equally well for some of them. We consider the general
recurrence (\ref{an-rr}). Then the OGF $A(z) := \sum_{n\ge1} a_n
z^n$ satisfies
\[
    A(z) = zA(z) + \frac{z}{1-pz} A\left(\frac{qz}{1-pz}\right)
    +B(z),
\]
where $B(z) := \sum_{n\ge1} b_n z^n$. Thus $\bar{A}(z) := (1-z)A(z)$
satisfies
\[
    \bar{A}(z) = B(z) + \frac{z}{1-z} \bar{A}
    \left(\frac{qz}{1-pz}\right),
\]
which after iteration gives
\[
    \bar{A}(z) = \sum_{j\ge0} q^{j(j-1)/2}
    \left(\frac{z}{1-z}\right)^j B\left(\frac{q^jz}
    {1-(1-q^j)z}\right).
\]
Thus
\begin{align}
    A(z) = \sum_{j\ge0}\frac{q^{j(j-1)/2} z^j}
    {(1-z)^{j+1}} B\left(\frac{q^jz}
    {1-(1-q^j)z}\right). \label{Az-exact}
\end{align}
Closed-form expressions can be derived from this; we omit the
details here.

\section{Variance of $Y_n$}

We derive in this section the asymptotics of the variance $Y_n$ (see
\eqref{Yn-rr}), which can be regarded as a very rough independent
approximation to $X_n$. We use an elementary approach (no complex
analysis being needed) here based on the recurrences of the central
moments and suitable tools of ``asymptotic transfer'' for the
underlying recurrence. The approach is, up to the development of
asymptotic tools, by now standard; see \cite{hwang03a,hwang02a}. The
same analysis provided here is also applicable to higher central
moments, which will be analyzed in the next section.

\subsection{Recurrence} For the variance of $Y_n$, we start with the
recurrence (\ref{Yn-rr}), which translates into the recurrence
satisfied by the moment GF $M_n(y) := \mathbb{E}
\left(e^{Y_ny}\right)$
\[
    M_n(y) = M_{n-1}(y) \sum_{0\le j<n}
    \pi_{n,j} M_{j}(y)\qquad(n\ge2),
\]
with $M_0(y)=1$ and $M_1(y)=e^y$, where $\pi_{n,j} :=
\binom{n-1}{j}q^jp^{n-1-j}$. This implies, with $\bar{M}_n(y) :=
e^{-\mu_ny}M_n(y)= \mathbb{E}\left(e^{(Y_n-\mu_n)y}\right)$, that
\begin{align}\label{Mny-rr}
    \bar{M}_n(y) = \bar{M}_{n-1}(y) \sum_{0\le j<n}
    \pi_{n,j}\bar{M}_{j}(y) e^{\Delta_{n,j}y}\qquad(n\ge2),
\end{align}
with $\bar{M}_n(y)=1$ for $n<2$, where
\[
    \Delta_{n,j} := \mu_{j}+\mu_{n-1}-\mu_n.
\]
Let $M_{n,m} := \mathbb{E}(Y_n-\mu_n)^m=\bar{M}_n^{(m)}(0)$,
$m\ge0$. Then from (\ref{Mny-rr}), we deduce that
\begin{align}
    M_{n,m}=M_{n-1,m}+\sum_{0\le j<n}\pi_{n,j}M_{j,m}
    + T_{n,m}, \label{Mnm}
\end{align}
where, for $m\ge1$,
\begin{align}
    T_{n,m} &= \sum_{\substack{k+\ell+h=m\\
    0\le k,\ell<m\\0\le h\le m}} \binom{m}{k,\ell,h}M_{n-1,k}
    \sum_{0\le j<n} \pi_{n,j} M_{j,\ell} \Delta_{n,j}^h\notag \\
    &=\sum_{0\le \ell<m}\binom{m}{\ell}
    \sum_{0\le j<n}\pi_{n,j}M_{j,\ell}\Delta_{n,j}^{m-\ell}
    \notag \\
    &\quad+\sum_{2\le k\le m-2}\binom{m}{k}M_{n-1,k}
    \sum_{0\le \ell\le m-k}\binom{m-k}{\ell}
    \sum_{0\le j<n}\pi_{n,j}M_{j,\ell}\Delta_{n,j}^{m-k-\ell}.
    \label{Tnm}
\end{align}
Note that since $M_{n,1}=0$ and $\sum_{0\le j<n}\pi_{n,j}
\Delta_{n,j}=0$, terms with $k=1$ and $k=m-1$ vanish.

In particular, the variance $\sigma_n^2=M_{n,2}$ satisfies
\[
    \sigma_n^2= \sigma_{n-1}^2 + \sum_{0\le j<n}
    \pi_{n,j} \sigma_{j}^2 + T_{n,2},
\]
where
\[
    T_{n,2} = \sum_{0\le j<n} \pi_{n,j} \Delta_{n,j}^2.
\]

\subsection{Asymptotics of $T_{n,2}$}

To proceed further, we first consider the asymptotics of
$\Delta_{n,j}$ for $j=qn+O(n^{2/3})$. By Taylor expansion and
\eqref{tfz-fe}, we have
\begin{align*}
    \tilde{f}(n)-\tilde{f}(n-1) &=\tilde{f}'(n)
    -\frac{\tilde{f}''(n)}{2}+\frac{\tilde{f}'''(n)}{3!}
    +O\left(\int_0^1 (1-t)^4\tilde{f}^{(4)}(n-t)\dd t\right)\\
    &=\tilde{f}'(n) -\frac{\tilde{f}''(n)}{2}+
    \frac{\tilde{f}'''(n)}{3!}
    +O\left(\tilde{f}\left(q^4n\right)\right),
\end{align*}
and
\[
    \tilde{f}''(n)-\tilde{f}''(n-1)=
    \tilde{f}'''(n)+O\left(\tilde{f}\left(q^4n\right)\right).
\]
These and (\ref{mun-asymp3}) yield
\[
\begin{split}
    \mu_n-\mu_{n-1} &=\tilde{f}'(n)-\frac{\tilde{f}''(n)}{2}
    +O\left(n^2\tilde{f}\left(q^4n\right)\right)\\
    &=\tilde{f}(qn) +O\left(n^2\tilde{f}\left(q^4n\right)\right),
\end{split}
\]
since $\tilde{f}\left(q^2n\right) =O\left(n^2(\log
n)^{-2}\tilde{f}\left(q^4n\right)\right)$. Then, for
$j=qn+x\sqrt{pqn}$, $|x|\le n^{1/6}$,
\begin{align}
    \Delta_{n,j}&=\mu_{j}-(\mu_{n}-\mu_{n-1}) \notag \\
    &=\tilde{f}(qn+x\sqrt{pqn})
    -\tilde{f}(qn)
    +O\left(n^2\tilde{f}\left(q^4n\right)\right)\notag\\
    &=\tilde{f}'(qn)x\sqrt{pqn}
    +O\left(n^2(1+x^2)\tilde{f}\left(q^4n\right)\right).
    \label{D-asymp}
\end{align}
Thus, by (\ref{ratio-f1}) and (\ref{ratio-f2}),
\begin{align}
    T_{n,2}&=\sum_{|x|\le  n^{1/6}}\pi_{n,j}
    \left|\tilde{f}'(qn)x\sqrt{pqn}+
    O\bigl(n^2\tilde{f}(q^4n)\bigr)\right|^2
    +O\left(\mu_n^2\sum_{|x|>n^{1/6}}\pi_{n,j}\right)
    \notag \\
    &=pqn\tilde{f}'(qn)^2\sum_{|x|\le n^{1/6}}\pi_{n,j}
    \left|x\right|^2+O\left(n^{9/2}\tilde{f}^2
    \left(q^4n\right)\right)\notag \\
    &=pqn\tilde{f}'(qn)^2+O\left(n^{9/2}\tilde{f}^2
    \left(q^4n\right)\right)\notag\\
    &\sim q^{-1} p  n^{-3}(\log_\kappa n)^4\tilde{f}(n)^2.
    \label{Tn2-asymp}
\end{align}
The next step then is to ``transfer" this estimate to the
asymptotics of the variance.

\subsection{Asymptotic transfer}

We now develop an asymptotic transfer result, which will be used to
compute the asymptotics of higher central moments of $Y_n$ (in
particular the variance).

More generally, we consider a sequence $\{a_n\}_{n\ge0}$ satisfying
the recurrence relation
\begin{align}
    a_n=a_{n-1}+\sum_{0\le j<n}\pi_{n,j}a_{j}+b_n
    \qquad(n\ge1), \label{an-rr}
\end{align}
where $a_0$ is finite (whose value is immaterial) and
$\{b_n\}_{n\ge1}$ is a given sequence.
\begin{lem}\label{lmm-sum-bn}
If $b_n\sim n^\beta(\log n)^\xi \tilde{f}(n)^\alpha$, where
$\alpha>0$, $\beta, \xi\in \mathbb{R}$. Then
\[
    \sum_{j\le n} b_j\sim \frac{n}{\alpha\log_\kappa n}\,b_n.
\]
\end{lem}
\begin{proof}
Define $\varphi(t) := t^\beta (\log t)^\xi\tilde{f}(t)^\alpha$. By
assumption, $b_n\sim \varphi(n)$. Since
$\tilde{f}'(t)/\tilde{f}(t)\sim t^{-1}\log_\kappa t$ (by
(\ref{ratio-f1})), we see that $\varphi'(t)>0$ for $t$ sufficiently
large, say $t\ge t_0>0$. Thus $\varphi(t)$ is monotonically
increasing for $t\ge t_0$. Then
\[
    \sum_{j\le n} b_j\sim\sum_{2\le j\le n}\varphi(j)
    = \int_2^n\varphi(t)\dd t+O(\varphi(n))
\]
By the asymptotic relation (\ref{ratio-f1}), we have
\[
\begin{split}
    \int_1^n\varphi(t)\dd t&=\int_1^n
    t^\beta (\log t)^\xi \tilde{f}(t)^\alpha\dd t \\ &\sim
    (\log\kappa)\int_1^n t^{\beta+1} (\log t)^{\xi-1}
    \tilde{f}(t)^{\alpha-1}\tilde{f}'(t) \dd t\\
    &\sim \frac{\log\kappa}\alpha
    \int_1^n t^{\beta+1}(\log t)^{\xi-1}
    \dd\tilde{f}(t)^\alpha \\ &=
    \frac{n \varphi(n)}{\alpha\log_\kappa n}+O\left(
    \int_1^n\frac{\varphi(t)}{t}\dd t\right),
\end{split}
\]
by an integration by parts. The integral on the right-hand side is
easily estimated as follows.
\begin{align*}
    \int_1^n\frac{\varphi(t)}{t}\dd t
    &=O\left(\varphi(qn) \int_1^{qn} t^{-1}
    \dd t + \varphi(n) \int_{qn}^n t^{-1} \dd t \right) \\
    &= O(\varphi(n)).
\end{align*}
This proves the lemma.
\end{proof}
\begin{prop}\label{prop-at}
If $b_n\sim  n^\beta (\log n)^\xi\tilde{f}(n)^\alpha$, where
$\alpha>1$, $\beta, \xi\in \mathbb{R}$, then
\begin{align}
    a_n=\left(1+O\left(n^{1-\alpha}(\log n)^{\alpha-1}
    \right)\right)\sum_{0\le j\le n} b_j
    \sim \frac{n}{\alpha \log_\kappa n}\,b_n. \label{anbn-at}
\end{align}
\end{prop}
\begin{proof}
We start with obtaining upper and lower bounds for $a_n$. Since
$b_n>0$ for sufficiently large $n$, say $n\ge n_0$. We may, without
loss of generality, assume that $b_n\ge0$ for $n\ge n_0$ (for,
otherwise, we consider $b_n':=b_n+\max_{j\le n_0}|b_j|$ and then
show the difference between the corresponding $a_n'$ and $a_n$ is of
order $\tilde{f}(n)$). Then $a_n\ge0$ and, by (\ref{an-rr}), we have
the lower bound
\[
    a_n \ge a_{n-1}+b_n \ge \sum_{0\le j\le n} b_j.
\]
Now consider the sequence
\[
    C_n := \frac{a_n}{\sum_{0\le j\le n} b_j}\ge 1\qquad(n\ge1),
\]
and the increasing sequence
\[
    C^*_n :=\max_{1\le  j\le  n}\{C_j\}\ge  1.
\]
Then we have the upper bound
\[
    a_k\le  C^*_n\sum_{0\le j\le k} b_j,
\]
for all $k\le n$.

In view of the recurrence relation (\ref{an-rr}), we have
\[
\begin{split}
    a_n&\le C^*_{n-1}\sum_{0\le j<n} b_j
    +C^*_{n-1}\sum_{0\le j<n}\pi_{n,j}
    \sum_{0\le \ell \le j} b_\ell+b_n \\
    &\le C^*_{n-1}\sum_{0\le j\le n} b_j
    +C^*_{n-1}\sum_{0\le j<n}\pi_{n,j}
    \sum_{0\le \ell \le j} b_\ell.
\end{split}
\]
By Lemma \ref{lmm-sum-bn} and Corollary \ref{lmm-ratio-f},
we see that there exist an absolute constant $K>0$ such that
\begin{equation}\label{sum_estim}
    \sum_{0\le j<n}\pi_{n,j}\sum_{0\le \ell \le j} b_\ell
    \le Kn^{-\alpha} (\log n)^\alpha\sum_{0\le j \le n}b_j
    =O\left(n^{1-\alpha}(\log n)^{\alpha-1} b_n\right).
\end{equation}
It follows that
\[
    a_n\le C^*_{n-1}\left(1+K n^{-\alpha}(\log n)^\alpha
    \right)\sum_{0\le j\le n} b_j.
\]
By our definition of $C_n$, we then have
\[
    C_n\le  C^*_{n-1}\left(1+Kn^{-\alpha}(\log n)^\alpha\right),
\]
and
\[
    C^*_n=\max\{C^*_{n-1},C_n\}\le  C^*_{n-1}
    \left(1+Kn^{-\alpha}(\log n)^\alpha\right).
\]
Consequently,
\[
    C^*_n\le  C^*_2\prod_{2\le j\le n}
    \left(1+Kj^{-\alpha}(\log j)^\alpha\right).
\]
Since the finite product on the right-hand side is convergent, we
conclude that the sequence $C^*_n$ is bounded, or more precisely,
\[
    C^*_n\le  C^*_2\prod_{j\ge 2}
    \left(1+Kj^{-\alpha}(\log j)^\alpha\right).
\]
Thus we obtain the upper bound
\[
    a_n\le  C\sum_{0\le j\le n}b_j,
\]
where $C>0$ is an absolute constant depending only on $p,
\alpha, \beta$ and $\xi$.

With this bound and defining $\tilde{a}_n := \sum_{0\le j<n}
\pi_{n,j}a_{j}$, we can rewrite the recurrence relation
(\ref{an-rr}) as
\begin{align}
    a_n &=a_{n-1}+\tilde{a}_n+b_n \notag \\
    &=\sum_{0\le j\le n} b_j+\sum_{0\le k\le n}
    \tilde{a}_k. \label{an-rr2}
\end{align}

Now by the estimate (\ref{sum_estim}), we see that
\begin{align*}
    \sum_{0\le j\le n}\tilde{a}_j&=O\left(1+
    \sum_{2\le j\le n} j^{1-\alpha}(\log j)^{\alpha-1}
    b_j \right)\\ &= O\left(1+ \varphi(qn)
    \sum_{2\le  j\le qn} j^{1-\alpha}(\log j)^{\alpha-1}
    +n^{1-\alpha}(\log n)^{\alpha-1}\sum_{qn<j\le  n}b_j
    \right),
\end{align*}
where $\varphi(t) := t^\beta (\log t)^\xi\tilde{f}(t)^\alpha$.
Observe that
\[
    \varphi(qn)\sim  n^{-\alpha}(\log n)^\alpha b_n\sim
    n^{-\alpha-1}(\log n)^{\alpha+1}
    \sum_{0\le j\le n}b_j.
\]
Thus
\[
    \sum_{0\le j\le n} \tilde{a}_j =O\left(n^{1-\alpha}
    (\log n)^{\alpha-1}\sum_{0\le j\le n}b_j\right).
\]
The proof of the Proposition is complete by substituting this
estimate into \eqref{an-rr2}.
\end{proof}

Denote by $[z^n]A(z)$ for the coefficient of $z^n$ in the Taylor
expansion of $A(z)$. Then, in terms of ordinary GFs, the asymptotic
transfer (\ref{anbn-at}) can be stated alternatively as
\[
    [z^n] A(z) \sim [z^n]\frac{B(z)}{1-z},
\]
(when $b_n$ satisfies the assumption of Proposition~\ref{prop-at}),
which means that the contribution from terms in the sum in
(\ref{Az-exact}) with $j\ge1$ is asymptotically negligible. Roughly,
since
\[
    b_{n,j} := [z^n]B\left(\frac{q^jz}{1-(1-q^j)z}\right)
    = n^{-1}\sum_{1\le \ell \le n} \binom{n}{\ell} q^{j\ell}
    (1-q^j)^{n-\ell} \ell b_\ell,
\]
we see that $b_{n,j} = O(q^j b_{\lfloor{q^jn}\rfloor})$. We can then
give an alternative proof of (\ref{anbn-at}) by using
(\ref{Az-exact}).

By (\ref{Tn2-asymp}) and a direct application of
Proposition~\ref{prop-at}, we obtain an asymptotic approximation to
the variance.
\begin{thm} The variance of $Y_n$ satisfies
\begin{align}
    \sigma_n^2 \sim C_\sigma n^{-2} (\log_\kappa n)^3 \tilde{f}(n)^2,
    \label{var-asymp}
\end{align}
where $C_\sigma := p/(2q)$.
\end{thm}
Thus we have
\[
    \frac{\mathbb{V}(Y_n)}{(\mathbb{E}(Y_n))^2}
    \sim C_\sigma n^{-2} (\log n)^3.
\]

Monte Carlo simulations (with $n$ a few hundred) suggest that the
ratio $\mathbb{V}(X_n)/\mathbb{V}(Y_n)$ grows concavely, so that one
would expect an order of the form $n^\beta(\log n)^{\xi}$ for some
$0<\beta<1$. But due to the complexity of the problem, we could not
run simulations of larger samples to draw more convincing
conclusions. Asymptotics of $\mathbb{V}(X_n)$ remains open.

\section{Asymptotic normality}

We prove in this section that $Y_n$ is asymptotically normally
distributed by the method of moments. Our approach is to start from
the recurrence (\ref{Mnm}) for the central moments and the
asymptotic estimate (\ref{var-asymp}) and then to apply inductively
the asymptotic transfer result (Proposition~\ref{prop-at}), similar
to that used in our previous papers \cite{hwang03a,hwang02a}.

\begin{thm} The distribution of $Y_n$ is asymptotically normal,
namely,
\[
    \frac{Y_n-\mu_n}{\sigma_n} \stackrel{d}{\to}
    \mathscr{N}(0,1),
\]
where $\stackrel{d}{\to}$ denotes convergence in distribution.
\end{thm}
We will indeed prove convergence of all moments.
\begin{proof}
By standard moment convergence theorem, it suffices to show that
\begin{align}
    M_{n,m}=\mathbb{E}(Y_n-\mu_n)^m
    \begin{cases}
        \displaystyle\sim \frac{(m)!}{(m/2)!2^{m/2}}\,
        \sigma_n^m,& \text{if $m$ is even},\\
        =o(\sigma_n^m),&\text{if $m$ is odd},
    \end{cases}\label{mm}
\end{align}
for $m\ge0$.

The cases when $m\le2$ having been proved above, we assume $m\ge3$.
By induction hypothesis, we have
\[
    M_{n,k}=O\left(\sigma_n^k\right)
    =O\left(n^{-k}(\log n)^{3k/2}\tilde{f}^k(n)\right),
\]
for $k<m$. Then, by (\ref{D-asymp}),
\[
\begin{split}
    \sum_{0\le j<n}\pi_{n,j}M_{j,\ell}\Delta_{n,j}^h
    &=O\left( M_{\tr{qn},\ell} n^{h/2}\tilde{f}(q^2n)^h\right)
    \\&=O\left(n^{-\ell}(\log n)^{3\ell/2}\tilde{f}(qn)^\ell
    n^{h/2}\tilde{f}(q^2n)^h\right) \\
    &=O\left(n^{-2\ell-3h/2}(\log n)^{5\ell/2+2h}
    \tilde{f}(n)^{\ell+h}\right).
\end{split}
\]
It follows (see (\ref{Tnm})) that, for $0\le \ell<m$,
\[
    \sum_{0\le j<n}\pi_{n,j}M_{j,\ell}\Delta_{n,j}^{m-\ell}
    =O\left( n^{-\ell/2-3m/2}(\log n)^{\ell/2+2m}\tilde{f}(n)^m
    \right);
\]
and, for $2\le k\le m-2$ and $0\le \ell\le m-k$,
\[
    M_{n-1,k}\sum_{0\le j<n}\pi_{n,j}M_{j,\ell}
    \Delta_{n,j}^{m-k-\ell}
    =O\left(n^{-\ell/2+k/2-3m/2}
    (\log n)^{\ell/2-k/2+2m}\tilde{f}(n)^m\right).
\]
Thus the main contribution to the asymptotics of $T_{n,m}$ will come
from the terms in the second group of sums in (\ref{Tnm}) with
$k=m-2$ and $\ell=0$. More precisely
\[
\begin{split}
    T_{n,m}& = \binom{m}{2}M_{n-1,m-2}T_{n,2}
    +O\left(n^{-3/2-m}(\log n)^{3(m+1)/2}
    \tilde{f}(n)^m\right).
\end{split}
\]
Note that $T_{n,2} \sim 2 n(\log_\kappa n)^{-1}\sigma_n^2 $;
see \eqref{Tn2-asymp}.

Thus if $m$ is even, then, by (\ref{Tn2-asymp}) and induction
hypothesis,
\begin{align*}
    T_{n,m} &\sim \frac{2 m!}{\bigl((m-2)/2\bigr)!2^{m/2}}
    \, n^{-1}(\log_\kappa n) \sigma_n^m\\
    &\sim  \frac{2m!}{\bigl((m-2)/2\bigr)!2^{m/2}}\,
    C_\sigma^{m/2} n^{-m-1}(\log_\kappa n)^{(3m/2+1)}\tilde{f}(n)^{m}.
\end{align*}
Applying the asymptotic transfer result (Proposition~\ref{prop-at})
with $\alpha=m$, we obtain
\begin{align*}
    M_{n,m}&\sim \frac{m!}{(m/2)!2^{m/2}}\,C_\sigma^{m/2}
    n^{-m} (\log n)^{3m/2}\tilde{f}(n)^{m}\\
    &\sim \frac{m!}{(m/2)!2^{m/2}}\,\sigma_n^m.
\end{align*}
In a similar manner, we can prove that if $m$ is odd, then
\[
    M_{n,m}=o(\sigma_n^m).
\]
This concludes the proof of (\ref{mm}) and the asymptotic normality
of $Y_n$.
\end{proof}

\section{The random variables $Z_n$}

We briefly consider the random variables defined recursively in
\eqref{Zn-rr}. The major interest is in understanding the robustness
of the asymptotic normality when changing the underlying probability
distribution from binomial to uniform.
\begin{thm} The mean value of $Z_n$ satisfies
\begin{align} \label{EZn}
    \mathbb{E}(Z_n) = C n^{-1/4} e^{2\sqrt{n}}
    \left(1+\frac9{16\sqrt{n}}+\frac{11}{1536n}+
    O\left(n^{-3/2}\right)\right),
\end{align}
where
\[
    C := \frac12\sqrt{\frac e\pi}\int_0^1 \left(1-\frac1v\right)
    e^{-v} \dd v \approx 0.06906\,46192\dots
\]
The limit law of the normalized random variables $Z_n/\mathbb{E}
(Z_n)$ is not normal
\[
    \frac{Z_n}{\mathbb{E}(Z_n)}
    \stackrel{d}{\to} Z,
\]
where the distribution of $Z$ is uniquely characterized by its
moment sequence and the GF $\zeta(y) := \sum_{m\ge1} \mathbb{E}
(Z^m)y^m/(m\cdot m!)$ satisfies the nonlinear differential equation
\begin{align} \label{zeta}
    y^2\zeta'' +y\zeta'-\zeta = y\zeta \zeta',
\end{align}
with $\zeta(0)=\zeta'(0)=1$.
\end{thm}
\begin{proof} (Sketch)
The proof the theorem is simpler and we sketch only the major steps.

\paragraph{Mean value.}
First, $\nu_n := \mathbb{E}(Z_n)$ satisfies the recurrence
\[
    \nu_n = \nu_{n-1}+\frac1n\sum_{0\le j<n}\nu_j \qquad(n\ge2),
\]
with $\nu_0=0$, and $\nu_1=1$. The GF $f(z)$ of $\mathbb{E}(Z_n)$
satisfies the differential equation
\[
    f' = \frac{2-z}{(1-z)^2}\,f + \frac{1}{1-z},
\]
with the initial condition $f(0)=0$. Surprisingly, this same
equation (and the same sequence $\{\nu_nn!\}_n$, which is
\href{http://oeis.org/A005189}{A005189} in Encyclopedia of Integer
sequences) occurs in the study of two-sided generalized Fibonacci
sequences; see \cite{fishburn88a,fishburn89a}. The first-order
differential equation is easily solved and we obtain the closed-form
expression
\begin{align*}
    f(z) = \frac{z}{1-z}+\frac{e^{1/(1-z)}}{1-z}\int_0^{1/(1-z)}
    \left(1-\frac1v\right)e^{-v}\dd t.
\end{align*}
From this, the asymptotic approximation \eqref{EZn} results from a
direct application of the saddle-point method (see
Flajolet and Sedgewick's book \cite[Ch.\ VIII]{flajolet09a});
see also \cite{fishburn89a}.

\paragraph{Asymptotic transfer.}
For higher moments and the limit law, we are led to consider the
following recurrence.
\begin{align}\label{an-bn-unif}
    a_n = a_{n-1}+\frac1n\sum_{0\le j<n}a_j + b_n\qquad(n\ge2),
\end{align}
with $a_0$ and $a_1$ given. For simplicity, we assume $a_0=b_0=0$.

\begin{prop} \label{prop-at2}
Assume $a_n$ satisfies \eqref{an-bn-unif}. If $b_n \sim c n^\beta
\nu_n^\alpha$, where $\alpha>1$ and $\beta\in\mathbb{R}$, then
\begin{align}\label{anunif-asymp}
    a_n \sim \frac{c}{\alpha-\alpha^{-1}}
    \, n^{\beta+1/2} \nu_n^\alpha.
\end{align}
\end{prop}
The proof is similar to that for Proposition~\ref{prop-at} and is
omitted.

\paragraph{Recurrence and induction.}
By Proposition~\ref{prop-at2} and the following recurrence relation
for the moment GF $Q(y) := \mathbb{E}(e^{Z_ny})$
\[
    Q_n(y) = \frac{Q_{n-1}(y)}{n}\sum_{0\le j<n} Q_j(y)\qquad
    (n\ge2),
\]
with $Q_0(y)=1$ and $Q_1(y)=e^y$, we deduce, by induction using
\eqref{anunif-asymp}, that
\[
    \mathbb{E}(Z_n^m) \sim \zeta_m \nu_n^m \qquad(m\ge1),
\]
where
\begin{align} \label{zeta-m}
    \zeta_m = \frac{1}{m-m^{-1}}\sum_{1\le j<m}\binom{m}{j}
    \frac{\zeta_j}{j}\, \zeta_{m-j}\qquad(m\ge2),
\end{align}
with $\zeta_0=\zeta_1=1$. It follows that the function $\zeta(y) :=
\sum_{m\ge1}\zeta_m y^m/(m\cdot m!)$ satisfies the differential
equation \eqref{zeta}.

\paragraph{Unique determination of the distribution.}
First, by a simple induction we can show, by \eqref{zeta-m}, that
$\zeta_m\le c m! K^m$ for a sufficiently large $K>0$. This is enough
for justifying the unique determination. Instead of giving the
details, it is more interesting to note that the nonlinear
differential equation \eqref{zeta} represents another typical case
for which the asymptotic behavior of its coefficients
($\mathbb{E}(Z^m)$ for large $m$) necessitates the use of the
psi-series method recently developed in \cite{chern12a}. We can show,
by the approach used there, that
\[
    \mathbb{E}(Z^m) = m\cdot m!\rho^{-m}\left(2+
    \frac{2}{3m^2}+ O\left(m^{-3}\right)\right),
\]
where $\rho>0$ is an effectively computable constant. Note that
there is no term of the form $m^{-1}$ in the expansion, a typical
situation when psi-series method applies; see \cite{chern12a}.
\end{proof}

\section*{Concluding remarks}
\label{sec:cr}

The approach we used in this paper is of some generality and is
amenable to other quantities. We conclude this paper with a few
examples and a list of some concrete applications where the scale
$n^{c\log n}$ also appears.

First, the expected number of independent sets in a random graph
(under the $\mathscr{G}_{n,p}$ model), as given in \eqref{Jn},
satisfies the recurrence ($\bar{J}_n := J_n+1$)
\[
    \bar{J}_n = \bar{J}_{n-1} + \sum_{0\le k<n}
    \binom{n-1}{k} q^k p^{n-1-k} \bar{J}_k \qquad(n\ge1),
\]
with $\bar{J}_0=1$. Thus the Poisson GF $\tilde{f}(z) := e^{-z}
\sum_{n\ge0}\bar{J}_n z^n/n!$ satisfies the equation
\[
    \tilde{f}'(z) = \tilde{f}(qz),
\]
with $\tilde{f}(0)=1$. The modified Laplace transform then satisfies
the functional equation
\[
    \tilde{f}^\star(s) = 1 + s\tilde{f}^\star(qs),
\]
which, by iteration, leads to the closed-form expression
\[
    \tilde{f}^\star(s) = \sum_{j\ge0} q^{j(j-1)/2} s^{j}.
\]
Thus all analysis as in Section~\ref{sec:ec} applies with $F$ and
$G$ there replaced by
\[
    F(s) := \sum_{j\in\mathbb{Z}} q^{j(j-1)/2} s^j,
    \quad G(u) := q^{(\{u\}^2+\{u\})/2}F\left(q^{-\{u\}}\right).
\]
We obtain for example
\[
    J_n= \frac{G\left(\log_\kappa\frac{n}
    {\log_\kappa n}\right)}{\sqrt{2\pi}}\cdot
    \frac{n^{1/\log\kappa+1/2}}{\log_\kappa n}
    \,\exp\left(\frac{\left(
    \log \frac{n}{\log_\kappa n}\right)^2}
    {2\log\kappa}\right)
    \left(1+O\left(\frac{(\log\log n)^2}{\log n}\right)\right).
\]

The same approach also applies to the pantograph equation
\[
    \Phi'(z) = a\Phi(qz)+ \Psi(z)\qquad(a>0),
\]
with $\Phi(0)$ and $\Psi(z)$ given, for $\Psi(z)$ satisfying
properties that can be easily imposed.

Other extensions will be discussed elsewhere. We conclude with some
other algorithmic, combinatorial and analytic contexts where
$n^{c\log n}$ appears.
\begin{itemize}

\item[--] Algorithmics: isomorphism testing (see
\cite{babai12a,grosek10a,huber11a,miller78a,rosenbaum12a}),
autocorrelations of strings (see \cite{guibas81a,rivals03a}),
information theory (see \cite{abu-mostafa86a}), random digital
search trees (see \cite{drmota09a}), population recovery
(see \cite{wigderson12a}), and asymptotics of
recurrences (see \cite{knuth66a,oshea04a});

\item[--] Combinatorics: partitions into powers (see
\cite{debruijn48a,mahler40a}; see also \cite{fredman74a} for a brief
historical account and more references), palindromic compositions
(see \cite{ji08a}), combinatorial number theory (see
\cite{cameron90a,lev01a}), and universal tree of minimum complexity
(see \cite{chung81a, goldberg68a});

\item[--] Probability: log-normal distribution (see
\cite{johnson94a}), renewal theory (see \cite{van-beek73a,
vardi81a}), and total positivity (see \cite{karlin96a});

\item[--] Algebra: commutative ring theory (see \cite{campbell99a}),
and semigroups (see \cite{kuzmin93a,reznykov06a,shneerson01a});

\item[--] Analysis: pantograph equations (see
\cite{iserles93a,kato71a}), eigenfunctions of operators (see
\cite{spiridonov95a}), geometric partial differential equations (see
\cite{demarchis10a}), and $q$-difference equations (see
\cite{adams31a,carmichael12a,di-vizio03a,
ramis92a,zhang99a,zhang12a}).
\end{itemize}
This list is not aimed to be complete
but to show to some extent the generality of the seemingly
uncommon scale $n^{c\log n}$; also it suggests the possibly
nontrivial connections between instances in various areas,
whose clarification in turn may lead to further development
of more useful tools such as those in this paper.

\bibliographystyle{abbrvnat}
\bibliography{MIS}
\end{document}